\documentclass[12pt, a4]{amsart}

\usepackage{amscd, amsmath, amssymb}
\usepackage{fullpage}
\usepackage{mathrsfs,epsfig}
\usepackage{float}

\theoremstyle{plain}
\newtheorem{thm}{Theorem}[section]

\theoremstyle{definition}

\newtheorem{ex}[thm]{Example}
\newtheorem{rem}[thm]{Remark}

\numberwithin{equation}{section}

 \usepackage{color}

\begin{document}

\title[Eigenvalues of Hammerstein integral equations]{An eigenvalue result for Hammerstein integral equations with sign
changing nonlinearities and functional terms} 

\date{}

\author[G. Infante]{Gennaro Infante}
\address{Gennaro Infante, Dipartimento di Matematica e Informatica, Universit\`{a} della
Calabria, 87036 Arcavacata di Rende, Cosenza, Italy}%
\email{gennaro.infante@unical.it}%

\author[G. A. Veltri]{Giuseppe Antonio Veltri}
\address{Giuseppe Antonio Veltri, Dipartimento di Matematica e Informatica, Universit\`{a} della
Calabria, 87036 Arcavacata di Rende, Cosenza, Italy}%
\email{veltrigiuseppea@gmail.com}%

\begin{abstract} 
We discuss, via a version of the Birkhoff-Kellogg theorem, the existence of
positive and negative eigenvalues of Hammerstein integral equations with sign-changing nonlinearities and functional terms. 
The corresponding eigenfunctions have a given norm that, in turn, provides a location for the eigenvalues. 
As an application, we study the solvability of parameter-dependent
boundary value problems for nonlocal ordinary differential equations. Two examples illustrate the applicability of the theory in the case of mixed and Dirichlet boundary conditions.
\end{abstract}



\maketitle

\section{Introduction}
In this manuscript
we study the existence of eigenvalues and eigenfunctions of the Hammerstein integral equation
\begin{equation}\label{HIE}
u(t) = \lambda \int_{0}^{1} k(t,s) f(s, u(s), H[u]) \, ds,
\end{equation}
where $k$ and $f$ are suitable functions and $H$ is a suitable functional.
Integral equations of a form similar to~\eqref{HIE}, under a variety of sign and growth conditions, have been object of interest by a number of authors, we mention here the papers \cite{bell1, bell2, correa, Luis, Bogdan, Goodrich3, Goodrich4, Goodrich6, Goodrich7, gi-math, Stanczy} and references therein.
A motivation for studying the integral equation~\eqref{HIE} is that it can occur when dealing with boundary value problems (BVPs) with nonlocal terms arising in the differential equation. For example, Sta\'nczy~\cite{Stanczy}
studied the existence of radial positive solutions of the BVP
\begin{equation} \label{nonlocal-stan}
\left\{
\begin{array}{l}
-\Bigl(\int_{\Omega} f(u(x))\,dx\Bigr)^p\Delta u(x)=\lambda (f(u(x)))^{q}, \quad x\in \Omega, \\
u=0, \quad \text{on }\partial \Omega ,%
\end{array}%
\right. 
\end{equation}%
where $\Omega \subset \mathbb{R}^n$ is either a ball ($n=1$) or an annulus. Note that the choice of $p=q=1$ and $f(u)=e^u$ in~\eqref{nonlocal-stan} yields
\begin{equation} \label{nonlocal-stan-p}
\left\{
\begin{array}{l}
-\Delta u(x)=\lambda  \dfrac{e^{u(x)}}{\int_0^1 e^{u(x)} \, dx}, \quad x\in \Omega, \\
u=0, \quad \text{on }\partial \Omega .%
\end{array}%
\right. 
\end{equation}%

In the particular case of a ball, in~\cite{Stanczy} the BVP~\eqref{nonlocal-stan} is rewritten firstly in the ordinary differential equation~(ODE) form
\begin{equation} \label{nonl-stan-ode}
\left\{
\begin{array}{l}
-u''(t)=\lambda  \dfrac{(f(u(t))^{q}}{\Bigl(\int_{0}^1 f(u(x))\,dx\Bigr)^p}, \quad t\in (0,1), \\
u'(0)=u(1)=0,%
\end{array}%
\right. 
\end{equation}%
and then the BVP~\eqref{nonl-stan-ode} is transformed into the Hammerstein integral equation
\begin{equation}\label{HIES}
u(t) = \lambda \int_{0}^{1} k(t,s) \frac{(f(u(s))^{q}}{\Bigl(\int_{0}^1 f(u(x))\,dx\Bigr)^p} \, ds.
\end{equation}
Positive solutions of~\eqref{HIES} are then obtained via the classical Krasnosel'ski\u\i{}-Guo fixed-point theorem. It is worth to mention that the BVP~\eqref{nonlocal-stan} arises in the description of various physical phenomena, for more details we refer the reader to the Introduction of~\cite{Stanczy}. 

Fija\l kowski and Przeradzki~\cite{Bogdan} and Engui\c{c}a and Sanchez~\cite{Luis} incorporate functional terms within the non-linearities and study, via topological methods, the solvability of BVPs of the form 
\begin{equation*}
\left\{
\begin{array}{l}
-\Delta u=f\bigl(u,\int_{\Omega} g(u(x))\,dx\bigr), \quad x\in \Omega, \\
u=0, \quad \text{on }\partial \Omega,%
\end{array}%
\right. 
\end{equation*}
where $\Omega\subset \mathbb{R}^n$ with $n\geq 2$ is a ball~\cite{Luis} or either a ball or an annulus~\cite{Bogdan}. Note that the interesting case of nonlinearities that~\emph{change sign}
is studied in~\cite{Luis, Bogdan}. We stress that BVPs with nonlinearities of indefinite sign have been studied recently by a number of authors, we refer the reader, for the local case,
to the recent monograph by Feltrin~\cite{Guglielmo} and, for the nonlocal case,
to the papers by Goodrich~\cite{Goodrich, Goodrich3, Goodrich5, Goodrich6}, and references therein.

Our approach relies on a version of a classical tool of Nonlinear Functional Analysis, the classical Birkhoff-Kellogg theorem, as stated in the book by Appell, De Pascale and Vignoli~\cite{applbook}. We mention that this version of the Birkhoff-Kellogg theorem has been used recently in~\cite{Inf-Luc} to investigate the solvability of third order parameter-dependent BVPs with associated sign-changing kernels. Here we prove, under suitable estimates on the nonlinearities involved, the existence of couples of positive and negative eigenvalues with associated eigenfunctions of fixed norm. We also furnish some bounds on the modulus of the eigenvalues that depend on the norm of the eigenfunctions. We stress that our theory is directly applicable to eigenvalue problems for second order nonlinear ODEs with functional terms, subject to homogeneous Dirichlet and mixed boundary conditions~(BCs), for details see Remark~\ref{rem1}. We also furnish two examples in which the occurring nonlinearities can be seen as a variation of the ones occurring in~\eqref{nonlocal-stan-p}, under the presence of a sign-changing weight function. In our examples we illustrate in details the applicability of our theory. Our results are, as far as we are aware, new and complement the results in~\cite{Luis, Bogdan}.

\section{Some eigenvalue results}
Let $r>0$ and let $(X,||\cdot||)$ be a normed space. We denote by
   $$ 
        B_r(X):=\{x \in X \mid ||x|| < r\},\quad \partial B_r(X):=\{x \in X \mid ||x|| = r\},
    $$
and recall the following version of the Birkhoff-Kellogg theorem, which will be useful in the sequel. 
\begin{thm}[\cite{applbook}, Theorem 10.2] \label{B-K}
Let $X$ be an infinite dimensional real Banach space, and let 
$\hat{F} : \overline{B}_\rho(X) \to X$ be a compact operator such that
\begin{equation}\label{cBK}
    \inf_{x \in \partial B_\rho(X)} \|\hat{F}(x)\| > 0.
\end{equation}
Then there exist $\lambda_+>0,\lambda_-<0 \in \Lambda_\rho(\hat{F})$, where 
\[
\Lambda_\rho(\hat{F}):=\{ \lambda \in \mathbb{R} \mid \text{there exists } x \in \partial B_\rho(X) \text{ such that } \hat{F}(x) = \lambda x \}.
\]
\end{thm}
In what follows we make use of the Banach space $C([0,1])$, endowed with the usual norm
\[
\|u\|_\infty := \sup_{t \in [0,1]} |u(t)|,
\]
and set, for the sake of brevity, ${B}_\rho := {B}_\rho(C([0,1]))$. With these ingredients we can now state our main result on the existence and localization of eigenvalues and eigenfunctions of the integral equation~\eqref{HIE}.
\begin{thm}\label{EHI}
Let \( \rho \in (0, +\infty) \) and assume that:
\begin{enumerate}
    \item \( k : [0,1]^2 \rightarrow [0, +\infty) \) is continuous.
    \item \( H : \overline{B}_\rho \rightarrow \mathbb{R} \) is continuous and there exist \( \underline{H_\rho}, \overline{H_\rho} \in \mathbb{R} \) such that 
    \[ \underline{H_\rho} \leq H[u] \leq \overline{H_\rho}, \text{ for every } u \in \overline{B_\rho}. \]
    \item \( f : \Pi_\rho \subset \mathbb{R}^3 \rightarrow \mathbb{R} \) is continuous, where \(\Pi_\rho := [0,1] \times [-\rho,\rho] \times [\underline{H_\rho}, \overline{H_\rho}]\).
    \item There exist two continuous functions \( \underline{f_\rho}, \overline{f_\rho} : [0,1] \rightarrow \mathbb{R} \) such that
    \[ \underline{f_\rho}(t) \leq f(t,u,v) \leq \overline{f_\rho}(t), \text{ for every } (t,u,v) \in \Pi_\rho.\]
    \item[$(5)$] Let 
    \begin{align*}
        \underline{F_\rho}(t) := \int_0^1 k(t,s)\underline{f_\rho}(s)\,ds,\quad
        \overline{F_\rho}(t) := \int_0^1 k(t,s)\overline{f_\rho}(s)\,ds,
    \end{align*}
    and assume that at least one of the following conditions holds:
    \begin{itemize}
        \item[$(5a)$] There exists $t_\rho \in [0,1]$ such that $\overline{F_\rho}(t_\rho)<0$.
        \item[$(5b)$] There exists $t_\rho \in [0,1]$ such that $\underline{F_\rho}(t_\rho)>0$.
    \end{itemize}
\end{enumerate}
Then there exist $\lambda_\rho^+>0$, $\lambda_\rho^-<0$ and $u^+,u^-\in\partial B_\rho$ such that $(\lambda_\rho^+,u^+)$ and $(\lambda_\rho^-,u^-)$ solve the integral equation~\eqref{HIE}.

Furthermore, assume that the couple $(\lambda_\rho,u_\rho)$ satisfies the integral equation~\eqref{HIE}. Then the following is true.
\begin{itemize}
    \item[$(6a)$] If $(5a)$ holds, then we have the estimate
    $|\lambda_\rho| \leq -\dfrac{\rho}{\overline{F_\rho}(t_\rho)}.$\vspace{1mm}

    \item[$(6b)$] If $(5b)$ holds, then we have the estimate $|\lambda_\rho| \leq \dfrac{\rho}{\underline{F_\rho}(t_\rho)}$.
\end{itemize}
\end{thm}
\begin{proof}
Consider the operator
\[
Tu(t) := \int_0^1 k(t,s) f(s,u(s),H[s]) \, ds.
\]
Due to the assumptions (1)--(3), the operator \( T \) maps $\overline{B}_\rho \subset C([0,1])$ into $C([0,1])$ and is compact. First of all note that, due to (1)--(4), for every $u \in \partial B_\rho$ and for every $t \in [0,1]$, we have
\[ \underline{F_\rho}(t)=\int_0^1 k(t,s) \underline{f_\rho}(s) ds \leq T u(t) \leq \int_0^1 k(t,s) \overline{f_\rho}(s) ds = \overline{F_\rho}(t). \]
Suppose now that $(5b)$ holds. If we take $u \in \partial B_\rho$, we obtain
\begin{equation}\label{ineq1}
0< \underline{F_\rho}(t_\rho) \leq T u(t_\rho) \leq ||T u||_\infty.    
\end{equation}
Note that the inequality $\underline{F_\rho}(t_\rho) \leq ||T u||_\infty$ does not depend on the chosen \( u \). Hence, by passing to the infimum, we obtain
\begin{equation*}
\inf_{u \in \partial B_\rho} ||T u||_\infty \geq \underline{F_\rho}(t_\rho) > 0.    
\end{equation*}
Assume now that $(5a)$ holds, then, for $u \in \partial B_\rho$, we have
\begin{equation}\label{ineq2}
    ||T u||_\infty \geq -T u(t_\rho) \geq -\overline{F_\rho}(t_\rho)>0.
\end{equation}
As in the previous case we obtain
\[
\inf_{u \in \partial B_\rho} ||T u||_\infty > 0.
\]
Since, both cases, the condition \eqref{cBK} is satisfied, the existence of pairs of eigenvalues and eigenfunctions with the desired properties follows by means of Theorem~\ref{B-K}.

For the eigenvalue estimates, we make use of the inequalities~\eqref{ineq1} and \eqref{ineq2}. In particular, if $(5a)$ holds, we obtain
\[
\rho=||u_\rho||_\infty = |\lambda_\rho| \|Tu_\rho\|_\infty {\geq} -|\lambda_\rho| \overline{F_\rho}(t_\rho)>0,
\]
which yields $|\lambda_\rho| \leq -\dfrac{\rho}{\overline{F_\rho}(t_\rho)}$.\\
In a similar way, if $(5b)$ holds, we get
\[
\rho=||u_\rho||_\infty = |\lambda_\rho| \|Tu_\rho\|_\infty {\geq} |\lambda_\rho| \underline{F_\rho}(t_\rho)>0,
\]
which implies that $|\lambda_\rho| \leq \dfrac{\rho}{\underline{F_\rho}(t_\rho)}$.
\end{proof}
\begin{rem}\label{remest}
In the special case of nonlinearities of the form $$f(t,u,H[u])=g(t)\ell(u,H[u]),$$ where $g$ is a continuous function, allowed to change sign, and $\ell$ is a continuous non-negative function, one may construct the functions $\underline{F_\rho}$ and $\overline{F_\rho}$ that appear in Theorem~\ref{EHI} in the following manner.
Assume that $\underline{\ell_\rho},\overline{\ell_\rho} \geq 0$ satisfy the inequality
    \[
    \underline{\ell_\rho} \leq \ell(u,v) \leq \overline{\ell_\rho}, \text{ for every } (u,v) \in [-\rho,\rho] \times [\underline{H_\rho},\overline{H_\rho}].
    \]
    Let us use the decomposition $g(t)=g_+(t)-g_-(t)$, where
    \[
    g_+(t):=\max\{0,g(t)\}\quad\text{and}\quad g_-(t):=-\min\{0,g(t)\}.
    \]
    Note that we get, for every $u \in \overline{B}_\rho$ and every $t \in [0,1]$,
    \begin{equation}\label{stima}
        \underline{f_\rho}(t):= \underline{\ell_\rho} g_+(t) - \overline{\ell_\rho} g_-(t) \leq g(t)\ell(u(t),H[u]) \leq \overline{\ell_\rho} g_+(t) - \underline{\ell_\rho} g_-(t) =: \overline{f_\rho}(t).
    \end{equation}
    From \eqref{stima} we obtain
    \begin{align}
    \underline{F_\rho}(t) &= \underline{\ell_\rho} \int_0^1 k(t,s) g_+(s) \, ds - \overline{\ell_\rho}\int_0^1 k(t,s) g_-(s) \, ds,\label{F_low gen}\\
    \overline{F_\rho}(t) &= \overline{\ell_\rho} \int_0^1 k(t,s) g_+(s) \, ds - \underline{\ell_\rho}\int_0^1 k(t,s) g_-(s) \, ds\label{F_upp gen}.
    \end{align}
\end{rem}
\begin{rem}\label{rem1}
Note that Theorem~\ref{EHI} is directly applicable to the solvability of BVPs for differential equations of the form
\[
-u''(t) = \lambda f(t, u(t), H[u]),
\]
subject to a variety of BCs. For illustrative reasons in the sequel, we focus on the special cases of the BCs
\begin{align}
    u(0)=u'(1)=0,\label{mbc}\\
    u(0)=u(1)=0.\label{dbc}
\end{align}
In both cases, we obtain an integral equation of the form~\eqref{HIE}, where $k$ is the associated Green's function, that is
\begin{equation}\label{greenmista}
    k(t,s)=
\begin{cases}
    t, &\text{if } 0\leq t\leq s\leq1, \\
    s, &\text{if } 0\leq s\leq t\leq1,
\end{cases}
\end{equation}
in the case of the BC~\eqref{mbc}, and
\begin{equation}\label{greendir}
k(t,s)=
\begin{cases}
    t(1-s), &\text{if } 0\leq t\leq s\leq1, \\
    s(1-t), &\text{if } 0\leq s\leq t\leq1,
\end{cases}
\end{equation}
in the case of BC~\eqref{dbc}.
\end{rem}
By means of the Remarks~\ref{remest} and~\ref{rem1}, we illustrate, in two concrete examples, the applicability of Theorem~\ref{EHI}.
\begin{ex}\label{hmp}
Consider the BVP
    \begin{equation}\label{hmpeq}
    \begin{cases}
    -u''(t) = \lambda  \dfrac{\sin\left(\frac{3}{2}\pi t\right)e^{u(t)}}{\int_0^1 e^{u(x)} \, dx}, & t\in(0,1),\\
    u(0)=u'(1)=0.
    \end{cases}
    \end{equation}
To the BVP~\eqref{hmpeq} we associate the integral equation
\begin{equation}\label{HIEmista}
    u(t) = \lambda \int_{0}^{1} k(t,s) \dfrac{\sin\left(\frac{3}{2}\pi s\right)e^{u(s)}}{\int_0^1 e^{u(x)} \, dx} \, ds,
\end{equation}
where $k(t,s)$ is as in \eqref{greenmista}.

We proceed in seeking two functions $\underline{f_\rho}$ and $\overline{f_\rho}$ that satisfy the statement of the Theorem~\ref{EHI}.

Fix \( u \in \overline{B}_\rho \). Firstly note that we have
\[ e^{-\rho} \leq \int_0^1 e^{u(x)} \, dx \leq e^\rho, \]
which yields
\[
e^{-\rho} \leq \frac{1}{\int_0^1 e^{u(x)} \, dx} \leq e^\rho.
\]
Furthermore, note that we have
\[
e^{-\rho} \leq e^{u(t)} \leq e^\rho, \text{ for every } t \in [0,1],
\]
and therefore we obtain that
\[
e^{-2\rho} \leq \frac{e^{u(t)}}{\int_0^1 e^{u(x)} \, dx} \leq e^{2\rho}, \text{ for every } t \in [0,1].
\]

We now utilize Remark~\ref{remest} by taking $g(t) = \sin\bigl(\frac{3}{2}\pi t\bigr)$. Note that the function $g$ is non-negative in $\left[ 0, 2/3 \right]$ and is non-positive in $\left[ 2/3,1 \right]$. This means that \eqref{F_low gen} and \eqref{F_upp gen} read
\begin{align}
    \underline{F_\rho}(t) &= e^{-2\rho} \int_0^{\frac{2}{3}} k(t,s) \sin\left(\frac{3}{2}\pi s\right) \, ds + e^{2\rho}\int_{\frac{2}{3}}^1 k(t,s) \sin\left(\dfrac{3}{2}\pi s\right) \, ds,\label{F_low}\\
    \overline{F_\rho}(t) &= e^{2\rho} \int_0^{\frac{2}{3}} k(t,s) \sin\left(\frac{3}{2}\pi s\right) \, ds + e^{-2\rho}\int_{\frac{2}{3}}^1 k(t,s) \sin\left(\dfrac{3}{2}\pi s\right) \, ds.\label{F_upp}
\end{align}
By direct calculation, we obtain that
\[
\int_0^{\frac{2}{3}} k(t,s) \sin\left( \dfrac{3}{2}\pi s \right) \,ds =
\begin{cases}
    \dfrac{4}{9\pi^2} \sin\left( \dfrac{3}{2}\pi t \right) + \dfrac{2t}{3\pi}, & \text{if}\quad 0 \leq t \leq \dfrac{2}{3},\vspace{1.5mm}\\
    \dfrac{4}{9\pi}, & \text{if}\quad \dfrac{2}{3} \leq t \leq 1,
\end{cases}
\]
and
\[
\int_{\frac{2}{3}}^1 k(t,s) \sin\left( \dfrac{3}{2}\pi s \right) \,ds =
\begin{cases}
    -\dfrac{2t}{3\pi}, & \text{if}\quad 0 \leq t \leq \dfrac{2}{3},\vspace{1mm} \\
    \dfrac{4}{9\pi^2} \sin\left( \dfrac{3}{2}\pi t \right) - \dfrac{4}{9\pi}, & \text{if}\quad \dfrac{2}{3} \leq t \leq 1,
\end{cases}
\]
which yields
\[
\underline{F_\rho}(t) =
\begin{cases}
    e^{-2\rho}\left[ \dfrac{4}{9\pi^2} \sin\left( \dfrac{3}{2}\pi t \right) + \dfrac{2t}{3\pi} \right] + e^{2\rho}\left[ -\dfrac{2t}{3\pi} \right], & \text{if}\quad 0 \leq t \leq \dfrac{2}{3},\vspace{1.5mm}\\
    e^{-2\rho}\left[ \dfrac{4}{9\pi} \right] + e^{2\rho}\left[ \dfrac{4}{9\pi^2} \sin\left( \dfrac{3}{2}\pi t \right) - \dfrac{4}{9\pi} \right], & \text{if}\quad \dfrac{2}{3} \leq t \leq 1,
\end{cases}
\]
and
\[
\overline{F_\rho}(t) =
\begin{cases}
    e^{2\rho}\left[ \dfrac{4}{9\pi^2} \sin\left( \dfrac{3}{2}\pi t \right) + \dfrac{2t}{3\pi} \right] + e^{-2\rho}\left[ -\dfrac{2t}{3\pi} \right], & \text{if}\quad 0 \leq t \leq \dfrac{2}{3},\vspace{1.5mm}\\
    e^{2\rho}\left[ \dfrac{4}{9\pi} \right] + e^{-2\rho}\left[ \dfrac{4}{9\pi^2} \sin\left( \dfrac{3}{2}\pi t \right) - \dfrac{4}{9\pi} \right], & \text{if}\quad \dfrac{2}{3} \leq t \leq 1.
\end{cases}
\]
Note that the functions $\underline{F_\rho}$ and $\overline{F_\rho}$ are continuous and $\underline{F_\rho}(0)=\overline{F_\rho}(0)=0$. Figure~\ref{fig1} illustrates the graphs of the functions $\underline{F_\rho}(t)$ and $\overline{F_\rho}(t)$ in the cases $\rho=10^{-1}$ and $\rho=1$.
\begin{figure}[h]
    \centering
    \includegraphics[width=1.0\linewidth]{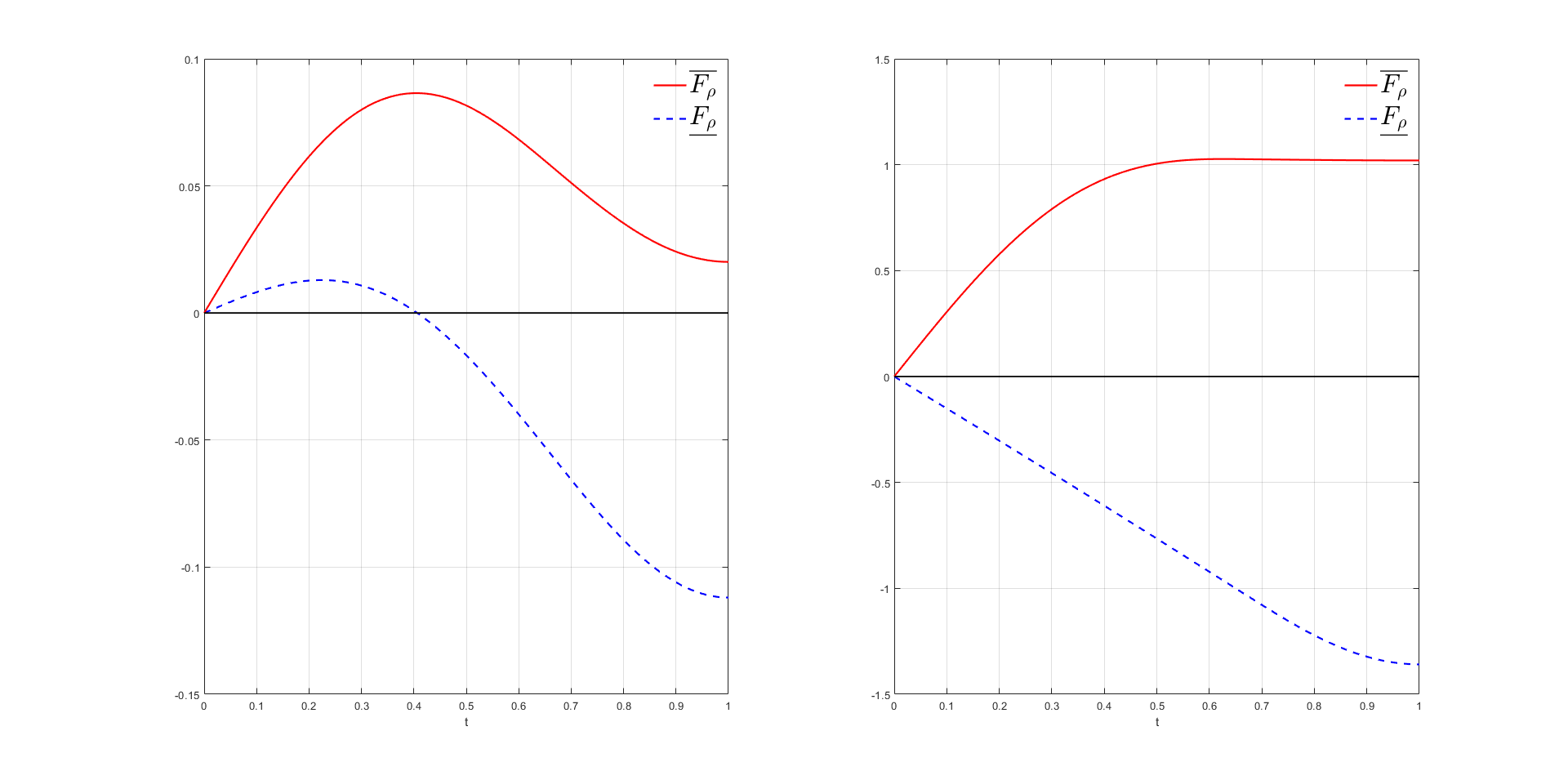}
    \caption{The plot on the left illustrates the case $\rho=10^{-1}$, and the one on the right displays the case $\rho=1$.}\label{fig1}
\end{figure}

We now seek the values of $\rho$ such that the function $\underline{F_\rho}$ achieves a positive maximum, so that the condition $(5b)$ holds. To do this, we follow the following scheme:
\begin{enumerate}
    \item We show that the function $\underline{F_\rho}$ is non-increasing in $\left[ 2/3,1 \right]$ and that $\underline{F_\rho}\left(2/3\right)<0$; this leads us to focus on the interval $\left[ 0,2/3 \right)$.

    \item We know that $\underline{F_\rho}(0)=0$, hence we study the interval $\left(0,2/3\right)$.

    \item For $\rho\geq \ln\sqrt[4]{2}$, we show that $\underline{F_\rho}$ is decreasing in $\left(0,2/3\right)$, and therefore a positive maximum is never attained. Nevertheless, for $0<\rho<\ln\sqrt[4]{2}$, we show that a positive maximum is attained.
\end{enumerate}
Let us start with the interval $\left[2/3,1\right]$. In this interval the function $\underline{F_\rho}$ reads as follows:
\[
\underline{F_\rho}(t) =
    e^{-2\rho}\left[ \dfrac{4}{9\pi} \right] + e^{2\rho}\left[ \dfrac{4}{9\pi^2} \sin\left( \dfrac{3}{2}\pi t \right) - \dfrac{4}{9\pi} \right],
\]
hence it has derivative
\[
\underline{F'_\rho}(t) =
    e^{2\rho}\left[ \dfrac{2}{3\pi} \cos\left( \dfrac{3}{2}\pi t \right)\right],
\]
which is always non-positive. This means that the maximum in the interval $\left[2/3,1\right]$ is attained at $2/3$ and is equal to
\[
\underline{F_\rho}\left( \dfrac{2}{3} \right) = \dfrac{4}{9\pi}\left(e^{-2\rho}-e^{2\rho} \right)<0, \quad \text{for every } \rho>0.
\]
Then, we seek a positive maximum for the function $\underline{F_\rho}$ in the interval $\left(0,2/3\right)$. In this interval, we have that
\[
\underline{F'_\rho}(t) = e^{-2\rho}\left[ \dfrac{2}{3\pi}\cos\left( 
\dfrac{3}{2}\pi t \right)+\dfrac{2}{3\pi} \right] + e^{2\rho}\left[ \dfrac{2}{3\pi} \right],
\]
which is $0$ if and only if $\cos\left( \dfrac{3}{2}\pi t \right) = e^{4\rho}-1$. Since $e^{4\rho}-1>0\geq-1$ for every $\rho>0$, the equality admits solutions if and only if $e^{4\rho}-1\leq 1$, i.e. if and only if $\rho\leq\ln\sqrt[4]{2}$. Actually, we have to exclude $\rho=\ln\sqrt[4]{2}$, since
\[
\cos\left( \dfrac{3}{2}\pi t \right) = 1 \iff t \in \left\{ 0, \dfrac{4}{3} \right\}\left(\not\subseteq\left( 0,\dfrac{2}{3} \right)\right).
\]
Since $\underline{F'_\rho}(t)=0$ can be solved if and only if $\rho<\ln\sqrt[4]{2}$, then for $\rho \geq \ln\sqrt[4]{2}$ it cannot be solved. Thus, for $\rho \geq \ln\sqrt[4]{2}$ we must have either $\underline{F'_\rho}(t)>0$ or $\underline{F'_\rho}(t)<0$ for every $t \in (0,2/3)$. Since $\underline{F_\rho}(2/3)<0$ for every $\rho>0$, the constant sign of $\underline{F_\rho'}$ yields
\[
\underline{F_\rho}(t) < 0 \quad \text{for every } t \in (0,2/3).
\]
Therefore, for $\rho\geq \ln\sqrt[4]{2}$ the function is always non-positive and its maximum value is $\underline{F_\rho}(0)=0$.

Now, let us search for the positive maximum whenever $0<\rho<\ln\sqrt[4]{2}\approx 0.173$. In this case, the maximum point can be obtained by setting
\[
\underline{F'_\rho}(t) \geq 0 \iff \cos\left( \dfrac{3}{2}\pi t \right) \geq e^{4\rho}-1.
\]
Furthermore, we note that $\underline{F'_\rho}(t_\rho) = 0,$ where
\[
t_\rho = \dfrac{2}{3\pi} \arccos(e^{4\rho}-1),
\]
$\underline{F'_\rho}(t) > 0$ for every $t\in\left(0,t_\rho\right)$, and $\underline{F'_\rho}(t) < 0$ for every $t\in\left(t_\rho,2/3\right)$. Hence, the only maximum point of $\underline{F_\rho}(t)$ in this interval is
\[
t_\rho =
\begin{cases}
    \dfrac{2}{3\pi} \arccos(e^{4\rho}-1), &\text{if}\quad 0<\rho<\ln\sqrt[4]{2},\\
    0, &\text{if}\quad \rho\geq\ln\sqrt[4]{2}.
\end{cases}
\]
Note that the value $\underline{F_\rho}(t_\rho)>0$ whenever $0<\rho<\ln\sqrt[4]{2}$. In Figure~\ref{fig2} we plot the graph of $\underline{F_\rho}(t_\rho)$ considered as a function of $\rho\in(0,+\infty)$. 
\begin{figure}[h!]
    \centering
    \includegraphics[width=0.5\linewidth]{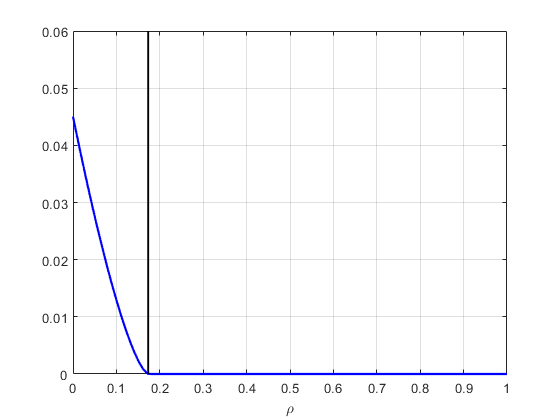}
    \caption{Graph of the functions $\underline{F_\rho}(t_\rho)$ and $\rho=\ln\sqrt[4]{2}$.}
    \label{fig2}
\end{figure}

Thus, for every $\rho \in (0, \ln\sqrt[4]{2})$, due to Theorem~\ref{EHI}, there exist $\lambda_\rho^+>0$, $\lambda_\rho^-<0$ and $u^+,u^-\in\partial B_\rho$ such that $(\lambda_\rho^+,u^+)$ and $(\lambda_\rho^-,u^-)$ solve the integral equation~\eqref{HIEmista}. 

Now, assume that the couple $(\lambda_\rho,u_\rho)$ satisfies the integral equation~\eqref{HIEmista}. Then, since $(5b)$ holds, due to Theorem~\ref{EHI} we obtain the localization 
\[
-a(\rho):=-\dfrac{\rho}{\underline{F_\rho}(t_\rho)}\leq\lambda_\rho \leq \dfrac{\rho}{\underline{F_\rho}(t_\rho)}=:a(\rho).
\]
Therefore, condition $(6b)$ holds. Figure~\ref{localmix} illustrates the obtained localization for the couples~$(\lambda_\rho,u_\rho)$.
\begin{figure}[h!]
    \centering
    \includegraphics[width=0.5\linewidth]{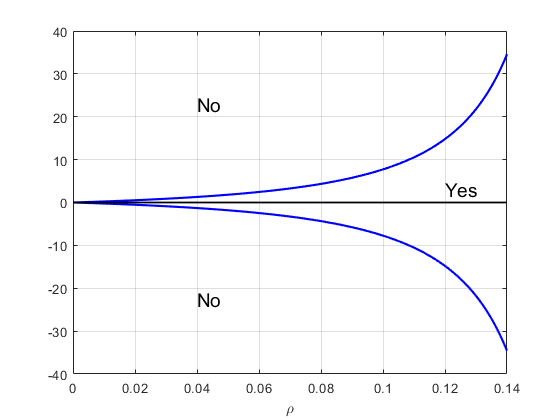}
    \caption{The region between $-a(\rho)$ and $a(\rho)$, for $\rho\leq0.14$.}
    \label{localmix}
\end{figure}
\end{ex}
We now illustrate the case of homogeneous Dirichlet BCs.
\begin{ex}
Consider the BVP
\begin{equation}\label{hdpeq}
\begin{cases}
-u''(t) = \lambda \sin\left( \dfrac{3}{2}\pi t \right) \dfrac{e^{u(t)}}{\int_0^1 e^{u(x)} \, dx}, & t\in(0,1),\\
u(0)=u(1)=0.
\end{cases}
\end{equation}
To the BVP~\eqref{hdpeq} we associate the integral equation
\begin{equation}\label{HIEdir}
    u(t) = \lambda \int_{0}^{1} k(t,s) \dfrac{\sin\left(\frac{3}{2}\pi s\right)e^{u(s)}}{\int_0^1 e^{u(x)} \, dx} \, ds,
\end{equation}
where $k$ is as in~\eqref{greendir}.

As in Example~\eqref{hmp}, we compute the two functions $\underline{F_\rho}$ and $\overline{F_\rho}$, that occur in~\eqref{F_low} and~\eqref{F_upp}, but with $k$ as in~\eqref{greendir}.
By direct calculation, we get that
\[
\int_0^{\frac{2}{3}} k(t,s) \sin\left( \dfrac{3}{2}\pi s \right) \,ds =
\begin{cases}
    \dfrac{4}{9\pi^2} \sin\left( \dfrac{3}{2}\pi t \right) + \dfrac{2t}{9\pi}, & \text{if}\quad 0 \leq t \leq \dfrac{2}{3},\vspace{1.5mm}\\
    \dfrac{4}{9\pi}(1-t), & \text{if}\quad \dfrac{2}{3} \leq t \leq 1,
\end{cases}
\]
and
\[
\int_{\frac{2}{3}}^1 k(t,s) \sin\left( \dfrac{3}{2}\pi s \right) \,ds =
\begin{cases}
    \dfrac{4t}{9\pi^2}-\dfrac{2t}{9\pi}, & \text{if}\quad 0 \leq t \leq \dfrac{2}{3},\vspace{1.5mm}\\
    \dfrac{4}{9\pi^2} \sin\left( \dfrac{3}{2}\pi t \right) + \dfrac{4}{9\pi}(t-1)+\dfrac{4t}{9\pi^2}, & \text{if}\quad \dfrac{2}{3} \leq t \leq 1.
\end{cases}
\]
In other words, we have that
\[
\underline{F_\rho}(t) =
\begin{cases}
    e^{-2\rho}\left[ \dfrac{4}{9\pi^2} \sin\left( \dfrac{3}{2}\pi t \right) + \dfrac{2t}{9\pi} \right] + e^{2\rho}\left[ \dfrac{4t}{9\pi^2}-\dfrac{2t}{9\pi} \right], & \text{if}\quad 0 \leq t \leq \dfrac{2}{3},\vspace{1.5mm}\\
    e^{-2\rho}\left[ \dfrac{4}{9\pi}(1-t) \right] + e^{2\rho}\left[ \dfrac{4}{9\pi^2} \sin\left( \dfrac{3}{2}\pi t \right) + \dfrac{4}{9\pi}(t-1)+\dfrac{4t}{9\pi^2} \right], & \text{if}\quad \dfrac{2}{3} \leq t \leq 1,
\end{cases}
\]
and
\[
\overline{F_{\rho}}(t) =
\begin{cases}
    e^{2\rho}\left[ \dfrac{4}{9\pi^2} \sin\left( \dfrac{3}{2}\pi t \right) + \dfrac{2t}{9\pi} \right] + e^{-2\rho}\left[ \dfrac{4t}{9\pi^2}-\dfrac{2t}{9\pi} \right], & \text{if}\quad 0 \leq t \leq \dfrac{2}{3},\vspace{1.5mm}\\
    e^{2\rho}\left[ \dfrac{4}{9\pi}(1-t) \right] + e^{-2\rho}\left[ \dfrac{4}{9\pi^2} \sin\left( \dfrac{3}{2}\pi t \right) + \dfrac{4}{9\pi}(t-1)+\dfrac{4t}{9\pi^2} \right], & \text{if}\quad \dfrac{2}{3} \leq t \leq 1.
\end{cases}
\]
Note that the functions $\underline{F_\rho}$ and $\overline{F_\rho}$ are continuous and $\underline{F_\rho}(0)=\underline{F_\rho}(1)=0$.

Figure~\ref{fig4} illustrates the graphs of the functions $\underline{F_\rho}$ and $\overline{F_\rho}$ in the cases $\rho=10^{-1}$ and $\rho=1$.
\begin{figure}[h!]
    \centering
    \includegraphics[width=0.9\linewidth]{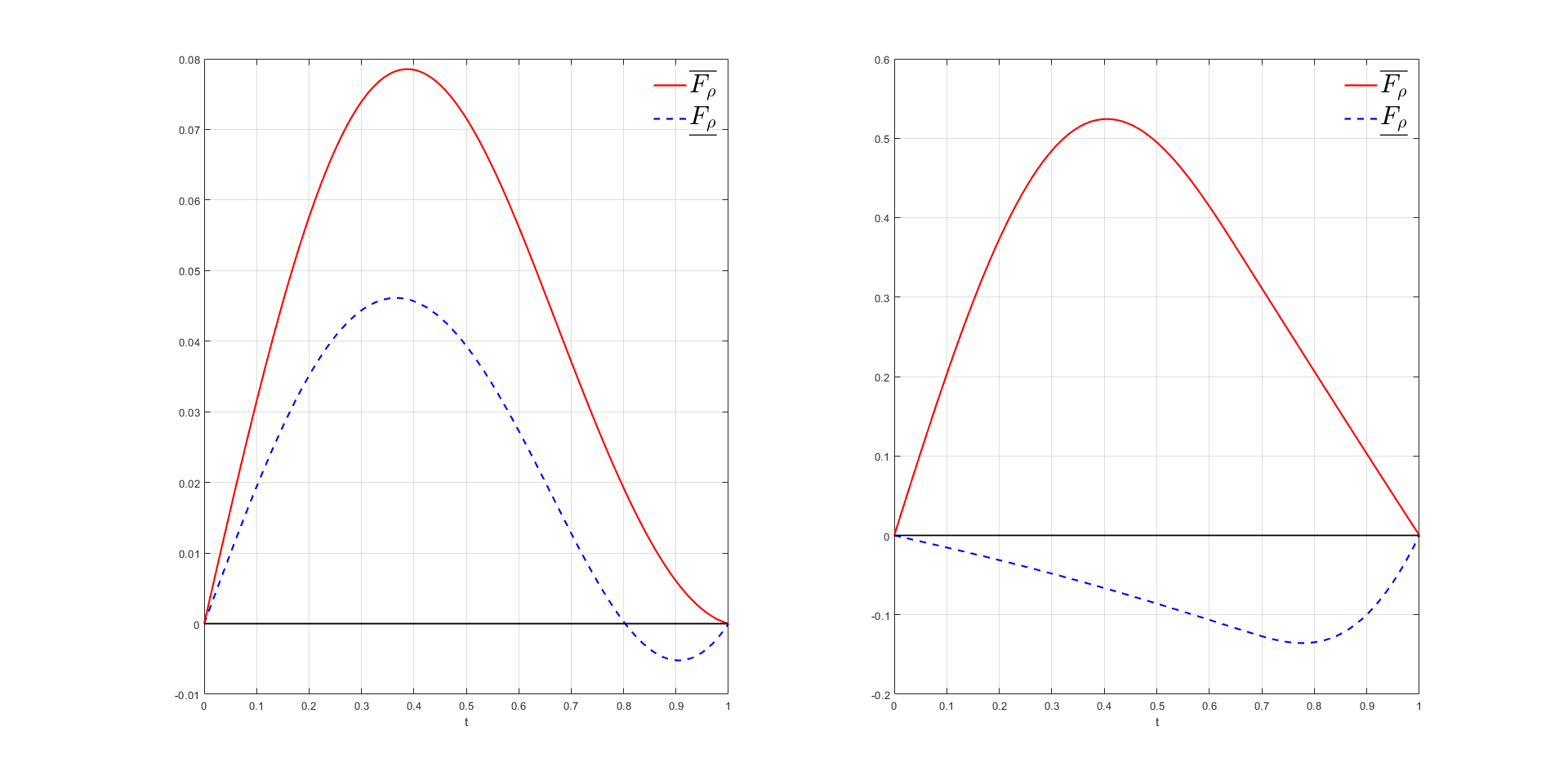}
    \caption{The plot on the left illustrates the case $\rho=10^{-1}$, and the one on the right displays the case $\rho=1$.}\label{fig4}
\end{figure}

We now seek the values of $\rho$ such that the function $\underline{F_\rho}$ achieves a positive maximum, so that the condition $(5b)$ holds. To do this, we apply the following scheme:
\begin{enumerate}
    \item We study $\underline{F_\rho}$ in the interval $\left[ 2/3,1 \right]$ and we get that for $\rho$ big enough, say $\rho\geq\rho_0$, the function does not attain a positive maximum in this interval. If $\rho<\rho_0$, the function restricted to the interval $[2/3,1]$ attains a positive maximum at $2/3$.

    \item We study $\underline{F_\rho}$ in the interval $\left(0,2/3\right)$ (we exclude $0$ since $\underline{F_\rho}(0)=0$). Again, we get that for $\rho$ big enough, say $\rho\geq\rho_1$, the function does not attain a positive maximum in this interval. Therefore, a positive maximum is attained in $(0,2/3)$ for~$\rho<\rho_1$.

    \item We put the previous steps together: a positive maximum is attained in $[0,1]$ if and only if $0<\rho<\max\{\rho_0,\rho_1\}$.
\end{enumerate}
Let us start with the interval $\left[2/3,1\right]$. In this interval the function $\underline{F_\rho}$ reads as follows:
\[
\underline{F_\rho}(t)=e^{-2\rho}\left[ \dfrac{4}{9\pi}(1-t) \right] + e^{2\rho}\left[ \dfrac{4}{9\pi^2} \sin\left( \dfrac{3}{2}\pi t \right) + \dfrac{4}{9\pi}(t-1)+\dfrac{4t}{9\pi^2} \right],
\]
hence it has derivative
\[
\underline{F'_\rho}(t)=-\dfrac{4}{9\pi}e^{-2\rho} + e^{2\rho}\left[ \dfrac{2}{3\pi} \cos\left( \dfrac{3}{2}\pi t \right) + \dfrac{4}{9\pi}+\dfrac{4}{9\pi^2} \right].
\]
In particular, we get that
\begin{equation}\label{equazione}
    \underline{F'_\rho}(t) = 0 \iff \cos\left( \dfrac{3}{2}\pi t \right) = \dfrac{2}{3}\left( e^{-4\rho}-1-\dfrac{1}{\pi} \right).
\end{equation}
One can solve this equation if and only if
\[
-\dfrac{3}{2} \leq e^{-4\rho} - 1 - \dfrac{1}{\pi} \leq \dfrac{3}{2}.
\]
The previous chain of inequalities is true for every $\rho>0$. Indeed,
\begin{align*}
\rho>0 \quad&\Rightarrow\quad 0 < e^{-4\rho} < 1\\
&\Rightarrow\quad -1-\dfrac{1}{\pi}<e^{-4\rho}-1-\dfrac{1}{\pi}<1-1-\dfrac{1}{\pi}=-\dfrac{1}{\pi}\\
&\Rightarrow\quad -\dfrac{2}{3}\left( 1+\dfrac{1}{\pi} \right)<\dfrac{2}{3}\left( e^{-4\rho}-1-\dfrac{1}{\pi} \right)<-\dfrac{2}{3\pi}.
\end{align*}
Furthermore, note that $-\dfrac{2}{3\pi}<0\leq1$. At the same time we have
\[
-\dfrac{2}{3}\left( 1+\dfrac{1}{\pi} \right) = \dfrac{-2\pi-2}{3\pi} \geq \dfrac{-2\pi-\pi}{3\pi} = -1.
\]
So, to summarize, the equation~\eqref{equazione} is solvable for every $\rho>0$. As in the previous example, we find the critical points by solving
\[
\underline{F'_\rho}(t) \geq 0 \iff \cos\left( \dfrac{3}{2}\pi t \right) \geq \dfrac{2}{3}\left( e^{-4\rho}-1-\dfrac{1}{\pi} \right).
\]
One finds that in $\left[2/3,1\right]$ there is a minimum at $t=\dfrac{4}{3}-\dfrac{2}{3\pi}\arccos\left( \dfrac{2}{3}\left( e^{-4\rho}-1-\dfrac{1}{\pi} \right) \right)$.\\
In other words, to seek a maximum in this interval, we need to evaluate the function at $2/3$ and $1$. We obtain that
\[
\underline{F_\rho}(1)=0 \quad\text{and}\quad \underline{F_\rho}\left(\dfrac{2}{3}\right) = \dfrac{4}{27\pi}e^{-2\rho}+\dfrac{8-4\pi}{27\pi^2}e^{2\rho}.
\]
Therefore, a positive maximum in $\left[ 2/3,1 \right]$ is attained if and only if
\[
\underline{F_\rho}\left(\dfrac{2}{3}\right) = \dfrac{4}{27\pi}e^{-2\rho}+\dfrac{8-4\pi}{27\pi^2}e^{2\rho} > 0 \iff e^{-4\rho}>\dfrac{\pi-2}{\pi} \iff \rho<\ln\sqrt[4]{\dfrac{\pi}{\pi-2}}\approx0.253.
\]

Now, we study $\underline{F_\rho}(t)$ in the interval $\left(0,2/3\right)$ and we have
\[
\underline{F_\rho}(t) = e^{-2\rho}\left[ \dfrac{4}{9\pi^2} \sin\left( \dfrac{3}{2}\pi t \right) + \dfrac{2t}{9\pi} \right] + e^{2\rho}\left[ \dfrac{4t}{9\pi^2}-\dfrac{2t}{9\pi} \right],
\]
thus
\[
\underline{F'_\rho}(t) = e^{-2\rho}\left[ \dfrac{2}{3\pi} \cos\left( \dfrac{3}{2}\pi t \right) + \dfrac{2}{9\pi} \right] + e^{2\rho}\left[ \dfrac{4}{9\pi^2}-\dfrac{2}{9\pi} \right].
\]
A direct calculation yields
\[
\underline{F'_\rho}(t) \geq 0 \iff \cos\left( \dfrac{3}{2}\pi t \right) \geq e^{4\rho} \left[ \dfrac{\pi-2}{3\pi} \right]-\dfrac{1}{3}.
\]
To have this inequality solvable, it's necessary to have that $e^{4\rho} \left[ \dfrac{\pi-2}{3\pi} \right]-\dfrac{1}{3} \in [-1,1]$. Actually, we have to exclude the case in which $e^{4\rho} \left[ \dfrac{\pi-2}{3\pi} \right]-\dfrac{1}{3} = 1$, since
\[
\cos\left( \dfrac{3}{2}\pi t \right) \geq 1 \iff \cos\left( \dfrac{3}{2}\pi t \right) = 1 \iff t\in\left\{0,\dfrac{4}{3}\right\}\left(\not\subseteq\left( 0,\dfrac{2}{3} \right)\right).
\]
Therefore, one has to search the values of $\rho$ such that
\[
-1\leq e^{4\rho}\left[ \dfrac{\pi-2}{3\pi} \right]-\dfrac{1}{3}<1.
\]
We observe that
\[
e^{4\rho} \left[ \dfrac{\pi-2}{3\pi} \right]-\dfrac{1}{3} > -\dfrac{1}{3}>-1 \quad \text{for every }\rho>0.
\]
Hence, we must solve the inequality $e^{4\rho} \left[ \dfrac{\pi-2}{3\pi} \right]-\dfrac{1}{3} < 1$ to get every possible value that $\rho$ may assume. In particular one gets that
\[
e^{4\rho} \left[ \dfrac{\pi-2}{3\pi} \right]-\dfrac{1}{3} < 1 \iff e^{4\rho} < \dfrac{4\pi}{\pi-2} \iff \rho < \ln\sqrt[4]{\dfrac{4\pi}{\pi-2}} \approx 0.599.
\]
To summarize, $\underline{F_\rho}$ attains a positive maximum in the whole interval $[0,1]$ if and only if $\rho$ is such that
\[
0 < \rho < \max\left\{ 
\ln\sqrt[4]{\dfrac{\pi}{\pi-2}},\ln\sqrt[4]{\dfrac{4\pi}{\pi-2}} \right\}=\ln\sqrt[4]{\dfrac{4\pi}{\pi-2}}.
\]
In this case, the maximum point is
\[
t_\rho = \dfrac{2}{3\pi} \arccos\left( e^{4\rho}\dfrac{\pi-2}{3\pi}-\dfrac{1}{3} \right).
\]
In conclusion, the maximum point is
\[
t_\rho =
\begin{cases}
    \dfrac{2}{3\pi} \arccos\left( e^{4\rho}\dfrac{\pi-2}{3\pi}-\dfrac{1}{3} \right), &\text{if}\quad 0 < \rho < \ln\sqrt[4]{\dfrac{4\pi}{\pi-2}},\\
    0, &\text{if}\quad \rho \geq \ln\sqrt[4]{\dfrac{4\pi}{\pi-2}}.
\end{cases}
\]
In Figure~\ref{fig5} we plot the graph of $\underline{F_\rho}(t_\rho)$, considered as a function of $\rho\in(0,+\infty)$.

\begin{figure}[hbt!]
    \centering
    \includegraphics[width=0.5\linewidth]{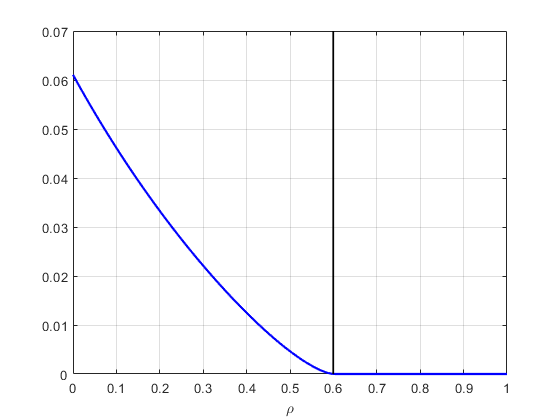}
    \caption{Graph of the functions $\underline{F_\rho}(t_\rho)$ and $\rho=\ln\sqrt[4]{\dfrac{4\pi}{\pi-2}}$.}\label{fig5}
\end{figure}

Thus, for every $\rho \in \Bigl(0, \ln\sqrt[4]{\dfrac{4\pi}{\pi-2}} \Bigr)$, due to Theorem~\ref{EHI}, there exist $\lambda_\rho^+>0$, $\lambda_\rho^-<0$ and $u^+,u^-\in\partial B_\rho$ such that $(\lambda_\rho^+,u^+)$ and $(\lambda_\rho^-,u^-)$ solve the integral equation~\eqref{HIEdir}. 

Now, assume that the couple $(\lambda_\rho,u_\rho)$ satisfies the integral equation~\eqref{HIEdir}. Then, since~$(5b)$ holds, due to Theorem~\ref{EHI} we obtain the localization 
\[
-a(\rho):=-\dfrac{\rho}{\underline{F_\rho}(t_\rho)}\leq\lambda_\rho \leq \dfrac{\rho}{\underline{F_\rho}(t_\rho)}=:a(\rho).
\]
Therefore, condition $(6b)$ holds. Figure~\ref{localdir} illustrates the obtained localization for the couples~$(\lambda_\rho,u_\rho)$.
\begin{figure}[H]
    \centering
    \includegraphics[width=0.5\linewidth]{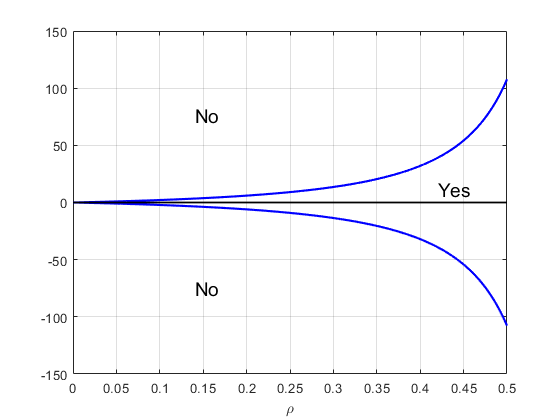}
    \caption{The region between $-a(\rho)$ and $a(\rho)$, for $\rho\leq0.5$.}
    \label{localdir}
\end{figure}
\end{ex}

\section*{Acknowledgements}
The authors would like to thank the anonymous Referees for the careful reading of the manuscript and the constructive comments.
G.~Infante is a member of the ``Gruppo Nazionale per l'Analisi Matematica, la Probabilit\`a e le loro Applicazioni'' (GNAMPA) of the Istituto Nazionale di Alta Matematica (INdAM), the UMI Group TAA  ``Approximation Theory and Applications'' and the ``The Research ITalian network on Approximation (RITA)''. G.~Infante
was partly funded by the Research project of MUR - Prin 2022 “Nonlinear differential problems with applications to real phenomena” (Grant Number:~2022ZXZTN2).

\end{document}